\documentclass{article}
\usepackage[T1]{fontenc}
\usepackage{amsfonts,amssymb,amsmath,amsthm}
\newcommand{\bN}{\mathbb{N}}

\newcommand{\bK}{\mathbb{K}}
\newcommand{\ord}{\mathrm{ord\,}}
\newtheorem{Theorem}{Theorem}
\newtheorem{Lemma}{Lemma}
\newtheorem{ex}{Example}

\title{On some regularity condition}
\author{Beata Gryszka and Janusz Gwo\'zdziewicz}
\begin{document}
\maketitle
\footnotetext{
      \begin{minipage}[t]{4.2in}{\small
       }
       \end{minipage}}

\begin{abstract}
Let $\bK$ be an uncountable field of characteristic zero and let $f$ be a function from $\bK^n$ to $\bK$.  We show that if the restriction of $f$ to every affine plane $L\subset\bK^n$ is regular, then $f$ is a regular function.
\end{abstract}

\section{Introduction}

In this paper open means always open in Zariski topology. 

Let $\bK$ be a field. 
We say that a function $f:\bK^n\to \bK$ is {\it regular at $a\in\bK^n$} 
if there exist polynomials $G$, $H\in\mathbb{K}[x_1,\dots,x_n]$ and an open 
neighborhood $U$ of $a$ in $\bK^n$ such that $H(x)\neq0$ and $f(x)=G(x)/H(x)$ for every $x\in U$. 
A function $f:\bK^n\to \bK$ is called {\it regular} if is regular at every point of $\bK^n$.  

We say that $f:\bK^n\to \bK$ has a {\it rational representation} $G/H$ 
if there exist polynomials $G$, $H\in\mathbb{K}[x_1,\dots,x_n]$ and a nonempty open 
set $U\subset\bK^n$ such that $H(x)\neq0$ and $f(x)=G(x)/H(x)$ for every $x\in U$. 

In order to keep the presentation on elementary level let us prove the following well-known fact 
(for more general case see \cite{To}).

\begin{Lemma}
 Every regular function $f$ from $\bK^n$ to $\bK$ can be written as a quotient 
$f=G/H$ of two polynomials $G,H\in\mathbb{K}[x_1,\dots,x_n]$ with nowhere vanishing 
denominator. 
\end{Lemma}
\begin{proof}
If $\bK$ is finite  then every $f:\bK^n\to \bK$ is a polynomial function.
Assume that $\bK$ is infinite. 
Let $f:\bK^n\to\bK$ be a regular function and $a_1$, $a_2$ be points in $\bK^n$. 
The function $f$ has rational representations $G_i/H_i$ 
such that $G_i$, $H_i$ are coprime and $f(a_i)=G_i(a_i)/H_i(a_i)$ for $i=1,2$.
Thus $G_1H_2=G_2H_1$ on a~non-empty open subset of $\bK^n$. 
Since the field $\bK$ is infinite, we get $G_1H_2=G_2H_1$ in the ring $\bK[x_1,\dots,x_n]$.
This ring is a unique factorization domain, hence there exists a non-zero constant $c\in\bK$ 
such that $H_2=cH_1$. Therefore $f(a_2)=G_1(a_2)/H_1(a_2)$. 
Since $a_2$ was arbitrary, we get $f=G_1/H_1$.
\end{proof}

Observe that if $\frac{FG}{FH}$ is a rational representation of $f$ then $\frac{G}{H}$ is also
a rational representation of $f$.  Hence we can always assume that a rational representation  
is a reduced fraction.

\section{Results}
The main results of this note are  Theorems \ref{Tw} and \ref{Tw:nieprzeliczane}.

\begin{Theorem}\label{Tw}
Let $\bK$ be a field of characteristic zero. Assume that a function $f:\bK^n\to \bK$ ($n\geq 2$) 
has a rational representation.  If for every vector plane $L\subset \bK^n$
the restriction $f|_L$ is regular, then  $f$ is regular at the origin. 
\end{Theorem}

%

Theorem~\ref{Tw} does not hold under a weaker assumption  ``regular at the origin''.
 
\begin{ex}
Let $f:\mathbb{R}^3\to\mathbb{R}$,  $f(x,y,z)=x$ if $(y-x^2)^2+(z-x^3)^2=0$ and $f(x,y,z)=0$ otherwise.
The restriction of $f$ to any vector plane is regular at the origin. However $f$ is not regular at the origin.  
\end{ex}

The following example shows that the assumption about a rational representation  
in Theorem \ref{Tw} is crucial. 

\begin{ex}
Let $\bK$ be a countable field. Consider the sequence of all pairwise different 
hyperplanes  $L_1, L_2,\dots$  in $\bK^n$ with equations $l_i=0$ for $i=1,2,\dots$. 
Then  $\bK^n$ is the union  $\bigcup_{i=1}^{\infty} L_i$. 
Let us define $$f=\sum_{i=1}^{\infty}
\prod_{j=1}^i l_j.$$ 

The function $f$ restricted to any affine hyperplane is a polynomial function.  However $f$ is not
regular.  Indeed, suppose that $f=G/H$. Let  $d$ be the degree of $G$. 
Then the restriction  of $G=fH$ to the hyperplane $L_{d+2}$ has degree at least $d+1$, 
but on the other hand, this degree is bounded by the degree of $G$. 
We arrived to contradiction.
\end{ex}

For uncountable fields we can say more.

\begin{Theorem}\label{Tw:nieprzeliczane}
Let $\bK$ be an uncountable field of characteristic zero and 
let $f$ be a function from $\bK^n$ to $\bK$ $(n\geq 2)$.
If for every affine plane  $L\subset \bK^n$, the restriction  $f|_L$ is regular,
then $f$ is regular.
\end{Theorem}

For $\bK=\mathbb{R}$ Theorem~\ref{Tw:nieprzeliczane} follows directly 
from the Bochnak-Siciak theorem (\cite{BS}) or from \cite[Theorem~6.1]{kollar2018curve}. 
For algebraically closed fields it follows from \cite[Theorem~8.3]{Pa} since in this case 
every regular function is a polynomial function.

The following classical example shows that in Theorem~\ref{Tw:nieprzeliczane} 
affine planes cannot be replaced by affine lines. 

\begin{ex}
Let $f:\mathbb{R}^2\to\mathbb{R}$ be a function defined as follows
$$ f(x,y)=
\begin{cases}
\frac{xy}{x^2+y^2} & \mbox{ for $(x,y)\neq (0,0)$} \\
 0                           & \mbox{ for $(x,y)=(0,0)$} 
\end{cases} .
$$
The restriction of $f$ to any affine line is regular,  but $f$ is not continuous, thus $f$ is not regular.   
\end{ex}

\section{Proofs}

\begin{Lemma}\label{lem0} 
Let $f:\bK^n\to\bK$ be a function with a rational representation $G/H$.
If the restriction of $f$ to every affine line is regular then 
$f(x)=G(x)/H(x)$ for all $x\in\bK^n$ such that $H(x)\neq0$.
\end{Lemma}

\begin{proof} 
Let $U$ be an open set such that $H(x)\neq0$ and $f(x)=G(x)/H(x)$ for $x\in U$.
Take any $a\in\bK^n$ such that $H(a)\neq0$ and let $L\subset\bK^n$ be an affine line passing though $a$
which has a nonempty intersection with $U$. Then by assumption there exist polynomial functions
$V,W:L\to \bK$ such that $f|_L=V/W$, $W(a)\neq0$. Comparing the rational representations 
$\frac{G|_L}{H|_L}$ and $\frac{V}{W}$ of $f|_L$ we get the polynomial equality 
$W\cdot G|_L=V\cdot H|_L$. 
Hence $f(a)=V(a)/W(a)=G(a)/H(a)$.
\end{proof}

Observe that every affine line in $\bK^n$ is contained in some vector plane. 
Hence any function that satisfies assumptions of Theorem~\ref{Tw} satisfies also 
assumptions of Lemma~\ref{lem0}.

\begin{Lemma}\label{lem1} 
If Theorem~\ref{Tw} holds for functions with  rational representations $G/H$, where
 $H\in\mathbb{K}[x_1,\dots,x_n]$ is irreducible, then it holds in full generality. 
\end{Lemma}

\begin{proof} 
We will proceed by induction on the number of irreducible factors of $H$. Suppose that Theorem~\ref{Tw} is true if $H$ has less than $s$ irreducible factors. Let us assume that $f$
with a rational representation $G/H$ fulfills the assumptions of Theorem~\ref{Tw} and $H$ has $s$ irreducible factors.

If  $H(0)\neq 0$, then by Lemma~\ref{lem0} the function $f$ is regular at zero. 
Otherwise there exists an irreducible factor  $H_1$ of  $H$ such that $H_1(0)=0$.

Write  $H=H_1^kH_2$, where $H_2$ is not divided by $H_1$. If  $H_2$ has positive degree, then  $H_1^kf$ and $H_2f$ fulfill the hypothesis and thus  $H_1^kf= V_1/W_1$, $H_2f=V_2/W_2$ 
in a neighborhood of $0\in\bK^n$.  Then

$$f=\frac{V_1}{H_1^kW_1}= \frac{V_2}{H_2W_2} 
$$
on non-empty open subset of $\mathbb{K}^n$. Hence $$H_1^kW_1V_2=H_2W_2V_1.$$  

Since $W_2(0)\neq 0$, we have that  $H_1$ does not divide $W_2$. Thus, by the above polynomial equality, we obtain that  $V_1=H_1^kV^{*}$  for some  $V^{*}\in\mathbb{K}[x_1,\dots,x_n]$ and then  
$f$ has a rational representation $V^{*}/W_1$. By Lemma~\ref{lem0} 
we get $f=V^{*}/W_1$ in a neighborhood of $0\in\bK^n$.

If  $H=aH_1^s$, where $a\in\mathbb{K}$, then by the hypothesis, we have
$$H_1^{s-1}f = \frac{G}{aH_1} = \frac{V}{W},
$$
on an open non-empty set. Similarly as before, we obtain  $G=aH_1V^{*}$ for some $V^{*}\in\mathbb{K}[x_1,\dots,x_n]$. Thus $f$ has a rational representation $V^*/H_1^{s-1}$.
Using the hypothesis we conclude that   $f$ is regular at zero.
\end{proof}

\begin{Lemma}\label{lem2}
Let $H:\bK^n\to \bK$ be an irreducible polynomial.  
Then after applying  some non-singular linear change of coordinates such that $\deg_{x_n} H=\deg H$  and  $H(0,\dots,0,x_n)$  has no multiple factors different from a power of $x_n$. 
\end{Lemma}

\begin{proof} 
Let $\tilde H(u_1,\dots,u_{n-1},v)=H(u_1v,\dots, u_{n-1}v,v)$
and $\hat H = v^{-\ord H}\tilde H$.   It is not difficult to check that $\hat H\in\mathbb{K}[u_1,\dots,u_{n-1},v]$ is irreducible. Let us write
$$ \hat H=h_mv^m+h_{m-1}v^{m-1}+\cdots+h_0,
$$
where $h_i\in \bK[u_1,\dots,u_{n-1}]$ for $i=0,\dots, m$ and $h_m\ne 0$.
Since $\hat H$ is irreducible,  its discriminant  $D=\mathrm{discrim}_v(\hat H)$ 
with respect to~$v$ is not equal to zero.   
Let $V\subset \bK^{n-1}$ be the set of points, where 
$h_m D\neq0$. Then, for every $(a_1,\dots,a_{n-1})\in V$, the polynomial  
$\hat H(a_1,\dots,a_{n-1},v)$ has no multiple factors. This implies that
$H(a_1v,\dots a_{n-1}v,v)=v^{\ord H}\hat H(a_1,\dots,a_{n-1},v)$ has no multiple factors different from a power of  $v$. Since $V$ is a non-empty open set, after a generic linear change of coordinates $x_i\mapsto x_i+a_ix_n$ for $i=1,\dots,n-1$ and $x_n\mapsto x_n$ we may assume that  $\deg_{x_n} H=\deg H$ and $H(0,\dots,0,x_n)$ has no multiple factors different from a power of  $x_n$. 
\end{proof} 

\begin{proof}[Proof of Theorem~\ref{Tw}] 
By Lemma~\ref{lem1} we can assume that $f$ has a rational representation  $G/H$, where $H$ 
is an irreducible polynomial and $G$, $H$ are coprime. We will prove that $H(0)\neq 0$.

Suppose to the contrary that $H(0)=0$. 
By Lemma~\ref{lem2}, we may assume
that $H$ is monic of degree $d>0$ with respect to~$x_n$,
and $ H(0,\dots,0,x_n)$ has no multiple factors different from a power of $x_n$. 

Consider the following division with remainder 
\begin{equation}\label{division}
G = qH+r
\end{equation}
in the ring $\bK[x_1,\dots,x_{n-1}][x_n]$.
Since $G$ and $H$ are coprime, the reminder  $r=r_kx_n^k+r_{k-1}x_n^{k-1}+\cdots+r_0$  is non-zero. Put $\tilde r_k\in \bK[x_1,\dots,x_{n-1}]$ to be the initial homogeneous form of~$r_k$. 

Let us  write $H(0,\dots,0,x_n)$ as a product
$$ \hat H_0(x_n)\cdots \hat H_s(x_n),$$
where $\hat H_0$ is a power of $x_n$ and if $s\geq 1$, then $\hat H_1,\dots, \hat H_s$ 
are relatively prime irreducible monic polynomials that do not vanish at zero.
By Hensel's lemma we obtain a factorization of $H$ in the ring $\bK[[x_1,\dots,x_{n-1}]][x_n]$.
$$  H= H_0\cdots  H_s,  
$$
where  $H_i(0,\dots,0,x_n)=\hat H_i(x_n)$ for $i=0,\dots ,s$.

\medskip
Given non-empty subset $I$ of $\{1,\dots, s\}$ define
$$A_I := \prod_{i\in I} H_i, \quad B_I:=\prod_{i\in \{0,\dots s\}\setminus I} H_i.
$$  
These are polynomials in the ring $\bK[[x_1,\dots,x_{n-1}]][x_n]$
of positive degrees such that $H=A_IB_I$.
Since $A_I$ and $B_I$ are not simultaneously elements of the ring $\bK[x_1,\dots,x_n]$, 
the polynomial $A_I\in\bK[[x_1,\dots,x_{n-1}]][x_n]$ has a non-zero coefficient
$t_I$ of a monomial $x_n^{m_I}$ for some $m_I<d$ which is not a polynomial i.e.  
$t_I\in\bK[[x_1,\dots,x_{n-1}]]\setminus \bK[x_1,\dots,x_{n-1}]$.  Hence 
there exists a non-zero homogeneous term $\tilde t_I$ of $t_I$ 
such that $\deg \tilde t_I>\deg_{x_1,\dots,x_{n-1}}H$.

\medskip
Let 
$$\pi=\tilde r_k \prod_{\emptyset \subsetneq I \subset \{1,\dots s\}} \tilde t_I .
$$

Take a point $a=(a_1,\dots,a_{n-1})\in\bK^{n-1}$ such that $\pi(a)\neq 0$ and consider a plane $L\subset \bK^n$ given by parametric equations $x_1=a_1u, \dots,x_{n-1}=a_{n-1}u, x_n=v$. 

Let us remark that for any power series 
$h=\sum c_{i_1,\dots,i_{n-1}}x_1^{i_1}\cdots x_{n-1}^{i_{n-1}}$ 
in $\bK[[x_1,\dots,x_{n-1}]]$ represented as a sum $h=\sum_{i=0}^{\infty} h_i$ of 
homogeneous polynomials  
$$h_i=\sum_{i_1+\cdots+i_{n-1}=i} c_{i_1,\dots, i_{n-1}}x_1^{i_1}\cdots x_{n-1}^{i_{n-1}}, 
\quad \mbox{for $i=0,1,\dots$}
$$
we have $h(a_1u,\dots,a_{n-1}u) =\sum_{i=0}^{\infty} h_i(a)u^i$
in the ring $\bK[[u]]$.

\medskip
By assumption there exist coprime polynomials $V,W:L\to\bK$ such that $W(0)\ne 0$ and
$\frac{G|_L}{H|_L} = \frac{V}{W}$ on a nonempty open subset of $L$.   Hence, we have
\begin{equation}\label{WGVH}
W\cdot  G|_L= V\cdot  H|_L.
\end{equation}
and
\begin{equation}\label{henselL}
H|_L=H_0|_L H_1|_L\cdots H_s|_L,
\end{equation}
where $H_0|_L,H_1|_L,\dots,H_s|_L$ are relatively prime  polynomials in $\mathbb{K}[[u]][v]$ and $H_i|_L$ are irreducible for $i=1,\dots,s$.

By the choice of the plane $L$, the polynomial $r|_L\in\bK[u,v]$ has 
a nonzero monomial $\tilde r_k(a)u^{\deg \tilde r_k}v^k$.  
It follows from~(\ref{division}) that $H|_L$ does not divide $G|_L$. 
Thus by~(\ref{WGVH}) the polynomial $W$ is a non-constant factor of $H|_L$. 
Since $W(0)\neq0$ and $H|_L(0)=0$, the polynomials $W$ and $H_0|_L$ are coprime.  
This implies that $s\geq 1$ and $W=A_I|_L$ for some subset $I$ 
such that $\emptyset \subsetneq I \subset \{1,\dots s\}$.

We proved that $A_I(a_1u,\dots,a_{n-1}u,v)\in\mathbb{K}[u,v]$.
It is impossible since $A_I(a_1u,\dots,a_{n-1}u,v)$ divides $H(a_1u,\dots,a_{n-1}u,v)$,
but in the first polynomial appears a monomial $\tilde t_I(a)u^{\deg \tilde t_I}v^{m_I}$
while the degree of the second polynomial with respect to~$u$ is not greater than  $\deg_{x_1,\dots,x_{n-1}}H$.  This contradiction finishes the proof.

\end{proof}

The above reasoning is based on the proof of \cite[Theorem 2]{kaltofen}, which is in fact the proof of some simple variation of the Bertini's theorem.

\medskip
Now we will prove Theorem \ref{Tw:nieprzeliczane}. We begin from two lemmas.

\begin{Lemma}\label{Podzial}
Let $\bK$ be an uncountable field and let  $\{A_i\}_{i\in I}$ be a countable division of the set $\bK^n$. 
Then there exists $i\in I$ such that for any polynomial  $f:\bK^n\to \bK$:  
if $f|_{A_i}=0$ then $f=0$.
\end{Lemma}
\begin{proof} 
We will proceed by induction on $n$.   

For $n=1$ any infinite set  $A_i$ from the family $\{A_i\}_{i\in I}$ fulfills the lemma, 
since every polynomial in one variable that vanishes on the infinite set is the zero polynomial. 

Assume that Lemma~\ref{Podzial} is true for $n-1$. Let $\{A_i\}_{i\in I}$ be a countable division 
of the set $\bK^n$. Without loss of generality, we may assume that  $I=\bN$. 
Define  $h: \bK\to\bN$ as follows: for $x\in\bK$ the value $h(x)$ is the smallest index  $i\in \bN$ such that the statement of Lemma~\ref{Podzial} holds 
for the division $\{A_i\cap (\bK^{n-1}\times \{x\})\}_{i\in\bN}$ of the set 
$\bK^{n-1}\times \{x\}$. By hypothesis the function $h$ is well-defined. Take any positive integer  $k$ such that the set $B_k=\{x\in\bK:h(x)=k\}$ is infinite. Let $f:\bK^n\to\bK$ be a polynomial such that  $f|_{A_k}=0$.   
For arbitrary  $b\in B_k$ the restriction 
$f|_{A_k\cap (\bK^{n-1}\times \{b\})}$ is equal to zero and thus
$f|_{\bK^{n-1}\times \{b\}}=0$. This implies that for every $a\in \bK^{n-1}$ the one-variable polynomial 
 $f|_{\{a\}\times\bK}$ vanishes on the infinite set $\{a\}\times B_k$.  
This means that $f|_{\{a\}\times\bK}=0$.  Since the point $a\in \bK^{n-1}$ was arbitrary, 
we obtain $f=0$. 
\end{proof}

\begin{Lemma}\label{Cramer}
Assume that   $f:\bK\to \bK$ is regular,  $f=G/H$,
$G(x)=b_0x^r+b_1x^{r-1}+\cdots+b_r$, 
$H(x)=x^s+c_1x^{s-1}+\cdots+c_s$ and $G/H$ is a reduced fraction.
Then, for any pairwise different elements $x_0,\dots, x_{r+s}$ of $\bK$ the  system of linear equations 
$$ B_0x_i^r+\cdots+B_r - f(x_i)(C_1x_i^{s-1}+\cdots +C_s) = f(x_i)x_i^s,\quad
\mbox{for $i=0,\dots, r+s$}
$$
with unknowns  $B_0,\dots, B_r,C_1,\dots,C_s$ has exactly one solution 
$$(B_0,\dots, B_r,C_1,\dots,C_s)=(b_0,\dots b_r,c_1,\dots,c_s).$$
\end{Lemma}

\begin{proof} 
From the equality $f=G/H$ it is clear that $(b_0,\dots b_r,c_1,\dots,c_s)$ is a solution of the above system of linear equations. Assume that
$(\tilde b_0,\dots \tilde b_r,\tilde c_1,\dots,\tilde c_s)$ is an arbitrary solution of this system. 
Then for the polynomials 
$\tilde G(x)=\tilde b_0x^r+\tilde b_1x^{r-1}+\cdots+\tilde b_r$, 
$\tilde H(x)=x^s+\tilde c_1x^{s-1}+\cdots+\tilde c_s$
we have $G(x_i)=f(x_i)H(x_i)$ for $i=0,\dots,r+s$ and thus 
$\tilde H G- H\tilde G=0$ on the set $\{x_0,\dots, x_{r+s}\}$.
Since the left hand side of this equality is a polynomial of degree  $\leq r+s$, we have the identity  
$\tilde H G- H\tilde G=0$. This implies that $\tilde G=G$ and $\tilde H=H$.
\end{proof}

We will call the pair $(r,s)$ of positive integers from Lemma \ref{Cramer}
the {\it type} of  $f$. 
The Cramer's Rule implies that for a regular function $f:\bK\to \bK$ of the type $(r,s)$ 
the coefficients of $G$ and $H$ are rational functions of 
$x_0,\dots, x_{r+s},f(x_0),\dots,$ $ f(x_{r+s})$
with integer coefficients.

\begin{Theorem}\label{Proste}
Let  $\bK$ be an uncountable field and let  $f:\bK^n\to \bK$ be a function such that  for any affine line  $L$ in $\bK^n$ that is parallel to one of the axes, the restriction  $f|_L$ is regular.  
Then $f$ has a rational representation.
\end{Theorem}

\begin{proof} We will proceed by induction on  $n$.  
For  $n=1$  the theorem is obvious.
Let $f:\bK^n\to \bK$ fulfills the assumption of the theorem. 
For $(x_1,\dots,x_{n-1},x_n)\in \bK^n$ we denote $x'=(x_1,\dots,x_{n-1})$. Let 
$$ A_{(r,s)} = \{x'\in \bK^{n-1}: f(x',\cdot):\bK\to \bK \mbox{ has the type $(r,s)$}\}.
$$
The family $\{A_{(r,s)}\}_{(r,s)\in \mathbb{N}^2}$ is a countable division of $\bK^{n-1}$. 
Take a pair $(r_0,s_0)\in\mathbb{N}^2$ which fulfills Lemma~\ref{Podzial}
and fix any  pairwise different elements $c_0,\dots, c_{r_0+s_0}$ of $\bK$. By the hypothesis, 
for every $i\in\{0,\dots,r_0+s_0\}$, the function
$f_i:=f(\cdot,c_i):\bK^{n-1}\to \bK$ has a rational representation. 
By the comment following  Lemma~\ref{Cramer},  $f$ 
restricted to the set $A_{(r_0,s_0)}\times  \bK$ has a rational representation $f=G/H$.
Fix $c\in \bK$.  By the hypothesis  $f|_{\bK^{n-1}\times\{c\}}$
has a form $G_1/H_1$. Since rational functions $G/H$ and $G_1/H_1$ have the same values on the set  $A_{(r_0,s_0)}\times\{c\}$, where $H_1\neq 0$, these functions are equal on $\bK^{n-1}\times\{c\}$. This implies that   $G/H$ is a rational representation of  $f$.
\end{proof}

In the language of \cite{Pa} Theorem~\ref{Proste} states that every separately regular function 
from $\bK^n$ to $\bK$ has a rational representation. Perhaps this theorem can be generalized 
to the product $K\times L$ of algebraic varieties as in~\cite{Pa}.


\medskip
Theorem~\ref{Tw:nieprzeliczane} follows easily from Theorems~\ref{Tw} and~\ref{Proste}.

\section{Acknowledgements}
This paper is a positive answer to the question of Wojciech Kucharz, which was presented to us in  private communication. We would like to thank Wojciech Kucharz for his valuable remarks.

\medskip
\noindent
{\small   Beata Gryszka\\
Institute of Mathematics\\
Pedagogical University of Cracow\\
Podchor\c{a}\.{z}ych 2\\
PL-30-084 Cracow, Poland\\
e-mail: bhejmej1f@gmail.com}

\medskip
\noindent
{\small   Janusz Gwo\'zdziewicz\\
Institute of Mathematics\\
Pedagogical University of Cracow\\
Podchor\c{a}\.{z}ych 2\\
PL-30-084 Cracow, Poland\\
e-mail: janusz.gwozdziewicz@up.krakow.pl}

\end{document}